\documentclass[11pt,reqno]{amsart}

\author[E. Abdo]{Elie Abdo}
\address[E. Abdo]
{
    Department of Mathematics \\
    American University of Beirut \\
    Beirut 1107-2020\\
    Lebanon.
} 
\email{ea94@aub.edu.lb}

\usepackage{xfrac}

\usepackage[margin=1in]{geometry}
\usepackage{amsmath, amsthm, amssymb, yhmath, xfrac, mathtools}
\usepackage[most]{tcolorbox}
\usepackage[normalem]{ulem}

\usepackage{times}
\usepackage{color}
\usepackage{hyperref}

\newtheorem{theorem}{Theorem}
\newtheorem{proposition}{Proposition}[section]
\newtheorem{lemma}[proposition]{Lemma}
\theoremstyle{remark}
\newtheorem{remark}{Remark}
\numberwithin{equation}{section}

\newcommand{\eps}{\epsilon}
\newcommand{\osc}{\operatorname{osc}}
\newcommand{\dv}{\mathrm{div}\,}

\newcommand{\Hst}{\sum_{i=1}^N\int_\Omega\Big(c_i\ln\frac{c_i}{c_i^*}-c_i+c_i^*\Big)dx}
\newcommand{\Pot}{\frac{\eps}{2}\int_\Omega|\nabla(\Phi-\Phi^*)|^2dx}
\newcommand{\Kin}{\frac{1}{2K}\int_\Omega|u|^2dx}
\newcommand{\HstAt}[1]{\sum_{i=1}^N\int_\Omega\Big(c_i(#1)\ln\frac{c_i(#1)}{c_i^*}-c_i(#1)+c_i^*\Big)dx}
\newcommand{\PotAt}[1]{\frac{\eps}{2}\int_\Omega|\nabla(\Phi(#1)-\Phi^*)|^2dx}
\newcommand{\KinAt}[1]{\frac{1}{2K}\int_\Omega|u(#1)|^2dx}
\newcommand{\Diss}{\sum_{i=1}^N D_i\int_\Omega c_i|\nabla(\ln c_i+z_i\Phi)|^2dx}

\usepackage{cite}

\usepackage{MnSymbol,wasysym}
\usepackage{accents}

\newcommand{\R}{\mathbb R}

\title[Generalized Logarithmic Sobolev inequalities and Applications to NPNS systems]{Generalized Logarithmic Sobolev inequalities on Bounded Domains and exponential decay to Boltzmann states for Nernst--Planck--Navier--Stokes systems}

\date{\today}
\begin{document}

\maketitle
\vspace{0.5cm}

\begin{abstract}
We derive logarithmic Sobolev inequalities on bounded domains for nonnegative functions that are prescribed on a portion of the boundary, and for Boltzmann-weighted reference measures. We use them to address the long-time dynamics of the Nernst--Planck--Navier--Stokes system with Dirichlet, blocking and mixed boundary conditions for the ionic concentrations and a Dirichlet--Neumann boundary condition for the electric potential. We prove that solutions converge exponentially fast to the Boltzmann steady states. When at least one species of each sign of valence is in contact with a reservoir, the asymptotic rate of convergence is independent of the dielectric constant and depends only on the domain, the boundary data, the masses of the blocked species, the valences, the diffusion coefficients, and the fluid viscosity.    
\end{abstract}

\vspace{0.5cm}

{\bf{Keywords.}} logarithmic Sobolev inequalities, Dirichlet boundary conditions, mixed boundary conditions, Nernst--Planck--Navier-Stokes equations, Boltzmann states, Poisson--Boltzmann equation, long-time dynamics

\vspace{0.5cm}

\tableofcontents

\section{Introduction}
The Nernst--Planck--Navier--Stokes (NPNS) system describes the electrodiffusion of ions in a fluid. The ionic concentrations $c_1,\dots,c_N$ are transported by the fluid velocity $u$, diffuse, and drift along the electric field $-\nabla\Phi$; the potential $\Phi$ is generated by the charge density $\rho=\sum_i z_ic_i$ through a Poisson equation with dielectric constant $\eps$; and the fluid is driven by the electric force $-K\rho\nabla\Phi$. The constant $\eps$ is proportional to the square of the Debye length, and in most physical regimes it is very small. This motivates the study of the system in the limit $\eps\to0$, and, for the long-time dynamics, the search for decay rates that do not deteriorate as $\eps\to0$.

In \cite{abdo2025logarithmic}, certain logarithmic Sobolev inequalities on bounded smooth domains have been established for the sake of studying the entropic behavior of solutions to Nernst-Planck equations when time blows up. Quantitatively, it has been shown that  
\begin{equation}\label{eq:AL}
\int_\Omega g\ln\frac{g}{\bar g}\,dx \le C\,\|\nabla\sqrt g\|_{L^2}^2, \qquad \bar g = \frac{1}{|\Omega|}\int_\Omega g\,dx,
\end{equation}
holds for nonnegative smooth functions $g$, provided that $\Omega \subset \R^d$, with $d=2,3,4$. In \eqref{eq:AL}, $C>0$ is a positive constant that depends only on $d$ and the diameter of $\Omega$ \cite[Theorem 1 and Remark 1]{abdo2025logarithmic}.The latter estimate \eqref{eq:AL} was then used to prove exponential convergence of solutions of the Nernst--Planck system with homogeneous Neumann boundary conditions to constant steady states, at a rate depending only on the domain and the diffusion coefficients. As pointed out in \cite[Remark 2]{abdo2025logarithmic}, this $\eps$-independence relies crucially on the boundary conditions. For other boundary conditions, the steady states are nonconstant Boltzmann states, and the approach of \cite{abdo2025logarithmic} partially breaks down. 

In this paper, we consider the physically relevant situation in which the ions are in contact with reservoirs on a portion of the boundary (Dirichlet boundary conditions for the concentrations), are blocked on the rest of it (no-flux boundary conditions), and the potential is prescribed on a portion of the boundary. The steady states are then Boltzmann states $c_i^*=Z_ie^{-z_i\Phi^*}$, where $\Phi^*$ solves a Poisson--Boltzmann equation. In general they are not constant: they exhibit boundary layers of width $\sqrt\eps$, in which $\nabla\Phi^*$ is of order $\eps^{-1/2}$. Our main result (Theorem \ref{thm:main}) states that, nevertheless, provided at least one species of each sign of valence is in contact with a reservoir, solutions converge to the Boltzmann state at an exponential rate whose asymptotic value is independent of $\eps$ but depends only on the domain, the boundary data, the masses of the blocked species, the valences, the diffusion coefficients and the viscosity. Our proof is based on the following main ingredients:
\begin{enumerate}
\item \emph{Logarithmic Sobolev inequalities with boundary conditions.} For nonnegative $g$ with $g=\gamma>0$ on a portion $\Gamma$ of $\partial\Omega$, we prove that 
\[
\int_\Omega\Big(g\ln\frac{g}{\gamma}-g+\gamma\Big)dx \le C\Big(1+\frac{\bar g}{\gamma}\Big)\|\nabla\sqrt g\|_{L^2}^2,
\]
where $C$ is a positive constant that depends only on $d$, $\Omega$ and $\Gamma$. We also prove weighted versions of this inequality and of \eqref{eq:AL}, in which the constant depends on the weight only through its oscillation.
\item \emph{An $\eps$-independent pointwise bound for the Boltzmann state.} By the maximum principle for the Poisson--Boltzmann equation, $\|\Phi^*\|_{L^\infty}$ is bounded by a constant that depends only on the boundary data and on the masses of the blocked species, although $\nabla \Phi^*$ is not uniformly bounded. When blocked species are present, the Boltzmann constants of these species depend on $\Phi^*$ itself, and the bound relies on the reservoir species of both signs. Since only the oscillation of the weights $e^{-z_i \Phi}$ enters the weighted logarithmic Sobolev inequalities, this is exactly the information needed.
\item \emph{A comparison with the instantaneous Boltzmann state.} At each time we compare the solution with the Boltzmann state generated by its own potential $\Phi(t)$. The Nernst--Planck dissipation is then exactly a weighted Fisher information, and a convexity argument shows that the relative entropy with respect to the instantaneous Boltzmann state controls the full free energy relative to the steady state.
\end{enumerate}

For the purpose of this paper, as in \cite{abdo2025logarithmic}, we do not address the issue of global regularity of solutions, and instead assume that we are working with global classical solutions with positive concentrations. For the state of the art of global well-posedness and regularity we refer the reader to \cite{constantin2019nernst, constantin2021interior, constantin2021nernst, constantin2022nernst}. Exponential convergence to thermal equilibrium under mixed boundary conditions has a long history for drift-diffusion models of semiconductors, going back to \cite{gajewski1986basic}; the entropy method for Fokker--Planck type equations is surveyed in \cite{arnold2001convex}; and the electroneutral limit $\eps\to0$ of NPNS is studied in \cite{constantin2021interior}.

This paper is organized as follows. In Section \ref{sec:LSI}, we prove logarithmic Sobolev inequalities with boundary conditions, unweighted and weighted. Then, in Section \ref{sec:app}, we use these inequalities to prove exponential convergence of solutions of the NPNS system to Boltzmann states, with $\eps$-independent asymptotic rate. Lastly, in Appendix \ref{sec:elliptic}, we prove the elliptic estimates used in Section \ref{sec:app}: an $L^\infty$ estimate for the mixed Dirichlet--Neumann problem and $\eps$-independent maximum principles for the Poisson--Boltzmann equation, with and without blocked species.

\section{Logarithmic Sobolev Inequalities with Boundary Conditions}\label{sec:LSI}

Throughout this paper, we denote by $d$ a dimension obeying $d \in\{2,3,4\}$ and by $\Omega\subset\mathbb R^d$ a bounded connected domain with smooth boundary. We denote by $C$ a positive constant whose dependency on the parameters is stated whenever needed. Let $\Gamma\subset\partial\Omega$ be a relatively open subset with positive surface measure. We recall the Friedrichs inequality
\begin{equation}\label{eq:friedrichs}
\int_\Omega h^2\,dx \le C\int_\Omega|\nabla h|^2\,dx 
\end{equation} that holds for all $h\in H^1(\Omega)$ with $h=0$ on $\Gamma$ for some positive constant $C>0$ that depends only on $\Omega$ and $\Gamma$.

The following lemma will be frequently used in the paper:

\begin{lemma}\label{lem:elem}
For all $a\ge0$, it holds that 
\begin{align}
&a\ln a-a+1 \le (a-1)^2, \label{eq:elem1}\\
&3(a-1)^2 \le (2a+4)(a\ln a-a+1). \label{eq:elem2}
\end{align}
If $a\ge e^2$, we have 
\begin{equation}\label{eq:elem3}
a\ln a-a+1 \ge a.
\end{equation}
Moreover, the function $a\mapsto a\ln a-a+1$ is convex on $[0,\infty)$, decreasing on $[0,1]$ and increasing on $[1,\infty)$.
\end{lemma}

\begin{proof} The inequality \eqref{eq:elem1} follows from $\ln a \le a - 1$, which gives $a \ln a \le a^2 - a$. The inequality \eqref{eq:elem3} follows from $\ln a \ge 2$. As for \eqref{eq:elem2}, we let $h(a) = (2a+4)(a\ln a - a + 1) - 3(a-1)^2$ for $a > 0$. A direct computation gives $h(1) = h'(1) = h''(1) = 0$ and $h'''(a) = 4(a-1)/a^2$. Hence $h''$ attains its minimum on $(0,\infty)$ at $a = 1$, so $h'' \ge 0$, and $h$ is convex with $h(1) = h'(1) = 0$. Therefore $h \ge 0$ on $(0,\infty)$, and \eqref{eq:elem2} at $a = 0$ follows by continuity. The last statement follows from the fact that the derivative of $a \mapsto a \ln a - a + 1$ is $\ln a$.
\end{proof}

\begin{proposition}\label{prop:dirLSI}
Let $\gamma>0$ and let $g\ge0$ be such that $g,\sqrt g\in H^1(\Omega)$ and $g=\gamma$ on $\Gamma$. Then
\begin{equation}\label{eq:dirLSI}
\int_\Omega\Big(g\ln\frac{g}{\gamma}-g+\gamma\Big)dx \le C\Big(1+\frac{\bar g}{\gamma}\Big)\|\nabla\sqrt g\|_{L^2}^2,
\end{equation}
where $\bar g = \frac{1}{|\Omega|}\int_\Omega g\,dx$ is the spatial average of $g$ over $\Omega$ and $C$ is a positive  constant that depends only on $d$, $\Omega$ and $\Gamma$.
\end{proposition}

\begin{proof}
Since $\int_\Omega g\ln\bar g\,dx=|\Omega|\bar g\ln\bar g$ and $\int_\Omega(-g+\gamma)\,dx=|\Omega|(\gamma-\bar g)$, we have
\[
\int_\Omega\Big(g\ln\frac{g}{\gamma}-g+\gamma\Big)dx = \int_\Omega g\ln\frac{g}{\bar g}\,dx + |\Omega|\Big(\bar g\ln\frac{\bar g}{\gamma}-\bar g+\gamma\Big).
\]
 The first term on the right-hand side can be  bounded by 
 \[ 
 \int_\Omega g\ln\frac{g}{\bar g}\,dx \le C \|\nabla\sqrt g\|_{L^2}^2
 \]
using \eqref{eq:AL}.  For the second term, we make use of the elementary inequality  \eqref{eq:elem1} with $a=\bar g/\gamma$ and obtain
\[
|\Omega|\Big(\bar g\ln\frac{\bar g}{\gamma}-\bar g+\gamma\Big) = |\Omega|\gamma\Big(\frac{\bar g}{\gamma}\ln\frac{\bar g}{\gamma}-\frac{\bar g}{\gamma}+1\Big) \le |\Omega|\gamma\Big(\frac{\bar g}{\gamma}-1\Big)^2 = \frac{|\Omega|}{\gamma}(\bar g-\gamma)^2.
\]
Now we write
\[
|\Omega|(\bar g-\gamma) = \int_\Omega(g-\gamma)\,dx = \int_\Omega(\sqrt g-\sqrt\gamma)(\sqrt g+\sqrt\gamma)\,dx.
\]
The function $\sqrt g-\sqrt\gamma$ belongs to $H^1(\Omega)$ and vanishes on $\Gamma$. By the Cauchy--Schwarz inequality and \eqref{eq:friedrichs}, we have
\begin{align*}
|\Omega|^2(\bar g-\gamma)^2 &\le \left(\int_\Omega(\sqrt g-\sqrt\gamma)^2dx\right)\left(\int_\Omega(\sqrt g+\sqrt\gamma)^2dx\right) \le C|\Omega|(\bar g+\gamma)\|\nabla\sqrt g\|_{L^2}^2.
\end{align*}
Dividing by $|\Omega|\gamma$, we deduce that
\[
\frac{|\Omega|}{\gamma}(\bar g-\gamma)^2 \le C\Big(1+\frac{\bar g}{\gamma}\Big)\|\nabla\sqrt g\|_{L^2}^2.
\]
Putting these estimates together, we obtain the desired estimate \eqref{eq:dirLSI}.
\end{proof}

We now consider weighted measures:

\begin{theorem}\label{thm:wLSI}
Let $V\in L^\infty(\Omega)$, $w=e^{-V}$, and  $\osc V=\sup V-\inf V$. Let $g\ge0$ and suppose that $f=g/w$ satisfies $f,\sqrt f\in H^1(\Omega)$. Then it holds that 
\begin{equation}\label{eq:wLSI}
\int_\Omega g\ln\Big(\frac{g}{w}\,\frac{\int_\Omega w\,dx}{\int_\Omega g\,dx}\Big)dx \le Ce^{\osc V}\int_\Omega w\,\big|\nabla\sqrt{g/w}\big|^2dx,
\end{equation}
where $C$ is a positive constant depending only on $d$ and the diameter of $\Omega$.
If, moreover, $g/w=Z$ on $\Gamma$ for some constant $Z>0$, then
\begin{equation}\label{eq:wdirLSI}
\int_\Omega\Big(g\ln\frac{g}{Zw}-g+Zw\Big)dx \le Ce^{\osc V}\Big(1+\frac{\int_\Omega g\,dx}{Z\int_\Omega w\,dx}\Big)\int_\Omega w\,\big|\nabla\sqrt{g/w}\big|^2dx,
\end{equation}
where $C$ is a positive constant that depends only on $d$, $\Omega$ and $\Gamma$.
\end{theorem}

\begin{proof}
We denote by $\bar f=|\Omega|^{-1}\int_\Omega f\,dx$ the average of $f$ over $\Omega$. For $t>0$, we define
\[
G(t) = \int_\Omega w\Big(f\ln\frac{f}{t}-f+t\Big)dx.
\]
We have $G'(t)=\int_\Omega w(1-f/t)\,dx$, which vanishes only at
\[
t_* = \frac{\int_\Omega wf\,dx}{\int_\Omega w\,dx} = \frac{\int_\Omega g\,dx}{\int_\Omega w\,dx}.
\]
Moreover, $G''(t)=t^{-2}\int_\Omega wf\,dx\ge0$. Therefore, $G$ attains its minimum at $t_*$. Since $\int_\Omega w(-f+t_*)\,dx=0$, it follows that 
\[
G(t_*) = \int_\Omega wf\ln\frac{f}{t_*}\,dx = \int_\Omega g\ln\Big(\frac{g}{w}\,\frac{\int_\Omega w\,dx}{\int_\Omega g\,dx}\Big)dx.
\]
As $f\ln(f/t)-f+t$ is nonnegative for every $t>0$ (by Lemma \ref{lem:elem}), we can  take $t=\bar f$, use the cancellation  $\int_\Omega(-f+\bar f)\,dx=0$,   apply the logarithmic Sobolev inequality \eqref{eq:AL} to $f$, and obtain
\begin{align*}
\int_\Omega g\ln\Big(\frac{g}{w}\,\frac{\int_\Omega w\,dx}{\int_\Omega g\,dx}\Big)dx = G(t_*) \le G(\bar f) &\le \sup w\int_\Omega\Big(f\ln\frac{f}{\bar f}-f+\bar f\Big)dx = \sup w\int_\Omega f\ln\frac{f}{\bar f}\,dx\\
&\le C \sup w \int_\Omega|\nabla\sqrt f|^2dx \le C\frac{\sup w}{\inf w}\int_\Omega w|\nabla\sqrt f|^2dx,
\end{align*}
yielding \eqref{eq:wLSI}, since $\sup w/\inf w=e^{\osc V}$. 

As for \eqref{eq:wdirLSI}, we let $t_*$ be as above and compute
\begin{align*}
\int_\Omega\Big(g\ln\frac{g}{Zw}-g+Zw\Big)dx &= \int_\Omega w\Big(f\ln\frac{f}{Z}-f+Z\Big)dx\\
&= \int_\Omega wf\ln\frac{f}{t_*}\,dx + \ln\frac{t_*}{Z}\int_\Omega wf\,dx + (Z-t_*)\int_\Omega w\,dx\\
&= \int_\Omega g\ln\Big(\frac{g}{w}\,\frac{\int_\Omega w\,dx}{\int_\Omega g\,dx}\Big)dx + Z\int_\Omega w\,dx\,\Big(\frac{t_*}{Z}\ln\frac{t_*}{Z}-\frac{t_*}{Z}+1\Big),
\end{align*} after making use of the identity $\int_{\Omega} t_* w\, dx = \int_{\Omega} fw\, dx$.  
The first term is controlled by \eqref{eq:wLSI}. For the second term, we use \eqref{eq:elem1} with $a=t_*/Z$, which gives
\[
Z\int_\Omega w\,dx\,\Big(\frac{t_*}{Z}\ln\frac{t_*}{Z}-\frac{t_*}{Z}+1\Big) \le \frac{(t_*-Z)^2}{Z}\int_\Omega w\,dx.
\]
In view of the identity
\[
(t_*-Z)\int_\Omega w\,dx=\int_\Omega w(f-Z)\,dx=\int_\Omega w(\sqrt f-\sqrt Z)(\sqrt f+\sqrt Z)\,dx,
\] the fact that $\sqrt f-\sqrt Z$ vanishes on $\Gamma$, the Cauchy--Schwarz inequality,  and the Friedrichs inequality \eqref{eq:friedrichs}, we estimate
\begin{align*}
(t_*-Z)^2\Big(\int_\Omega w\,dx\Big)^2 &\le \left(\int_\Omega w(\sqrt f-\sqrt Z)^2dx\right) \left(\int_\Omega w(\sqrt f+\sqrt Z)^2dx\right)\\
&\le C\sup w \left(\int_\Omega|\nabla\sqrt f|^2dx\right) \left(\int_\Omega w(f+Z)\,dx
\right)\\
&\le Ce^{\osc V}\left(\int_\Omega w|\nabla\sqrt f|^2dx\right)\,(t_*+Z)\left(\int_\Omega w\,dx\right).
\end{align*}
Dividing both sides by $Z\int_\Omega w\,dx$, we infer that
\[
\frac{(t_*-Z)^2}{Z}\int_\Omega w\,dx \le Ce^{\osc V}\Big(1+\frac{t_*}{Z}\Big)\int_\Omega w|\nabla\sqrt f|^2dx.
\]
Since $t_*/Z=\int_\Omega g\,dx/(Z\int_\Omega w\,dx)$, adding this bound to \eqref{eq:wLSI} yields \eqref{eq:wdirLSI}. 
\end{proof}

\section{Application: Long-Time Behavior of Solutions to Electrodiffusion Models}\label{sec:app}

\subsection{The model and its Boltzmann states}

We consider the Nernst--Planck--Navier--Stokes system
\begin{align}
&\partial_tc_i + u\cdot\nabla c_i = D_i\,\dv(\nabla c_i+z_ic_i\nabla\Phi), \qquad i=1,\dots,N, \label{eq:NP}\\
&-\eps\Delta\Phi = \rho = \sum_{i=1}^Nz_ic_i, \label{eq:Poisson}\\
&\partial_tu + u\cdot\nabla u - \nu\Delta u + \nabla p = -K\rho\nabla\Phi, \qquad \dv u=0, \label{eq:NS}
\end{align}
in $\Omega\times(0,\infty)$. The $D_i>0$ are the diffusivity coefficients, the $z_i\in\mathbb R$ are the valences, $\eps>0$ is the dielectric constant, $\nu>0$ is the kinematic viscosity and $K>0$ is a coupling constant.

Let $\Sigma_D\subset\partial\Omega$ be a relatively open subset with positive surface measure and $\Sigma_N=\partial\Omega\setminus\Sigma_D$, which may be empty. We denote by $H^1_{\Sigma_D}(\Omega)$ the space of functions in $H^1(\Omega)$ whose trace vanishes on $\Sigma_D$. For each $i$, let $\Gamma_i\subset\Sigma_D$ be relatively open, and either empty or of positive surface measure. We impose the boundary conditions
\begin{equation}\label{eq:bc}
\begin{gathered}
\Phi=W \ \text{on } \Sigma_D, \qquad \partial_n\Phi=0 \ \text{on } \Sigma_N, \qquad u=0 \ \text{on } \partial\Omega,\\
c_i=\gamma_i \ \text{on } \Gamma_i, \qquad \partial_nc_i+z_ic_i\partial_n\Phi=0 \ \text{on } \partial\Omega\setminus\Gamma_i,
\end{gathered}
\end{equation}
where $W$ and $\gamma_i>0$ are smooth, time-independent functions. Thus the $i$-th species is in contact with a reservoir on $\Gamma_i$ and is blocked on the rest of the boundary; we call it a blocking species if $\Gamma_i=\emptyset$. We assume the equilibrium compatibility condition
\begin{equation}\label{eq:compat}
\gamma_i\,e^{z_iW} \equiv Z_i \ \text{ is a positive constant on } \Gamma_i \qquad \text{whenever } \Gamma_i\neq\emptyset.
\end{equation}

These boundary conditions include, in particular:
\begin{enumerate}
    \item blocking boundary conditions for all species together with a Dirichlet boundary condition for the potential on the whole boundary, which is the classical Poisson--Boltzmann setting; 
    \item Dirichlet boundary conditions for all species; 
    \item selective boundary conditions, where some species are in contact with reservoirs and others are blocked (boundary conditions of this type are considered in \cite{constantin2019nernst});
    \item and their combinations on different portions of the boundary.
\end{enumerate} 

If the $i$-th species is blocking, its mass is conserved. Indeed, integrating \eqref{eq:NP} and using $\dv u=0$, $u=0$ on $\partial\Omega$ and $\partial_nc_i+z_ic_i\partial_n\Phi=0$ on $\partial\Omega$, we have
\begin{equation}\label{eq:mass}
\frac{d}{dt}\int_\Omega c_i\,dx = -\int_\Omega\dv(uc_i)\,dx + D_i\int_{\partial\Omega}(\partial_nc_i+z_ic_i\partial_n\Phi)\,dS = 0.
\end{equation}

The Boltzmann state associated with \eqref{eq:NP}--\eqref{eq:bc} is given by $u^*=0$ and
\begin{equation}\label{eq:Boltz}
c_i^* = Z_i\,e^{-z_i\Phi^*}, \qquad Z_i \text{ given by \eqref{eq:compat} if } \Gamma_i\neq\emptyset, \qquad Z_i = \frac{\int_\Omega c_i(0)\,dx}{\int_\Omega e^{-z_i\Phi^*}dx} \ \text{ if } \Gamma_i=\emptyset,
\end{equation}
where $\Phi^*$ solves the Poisson--Boltzmann equation
\begin{equation}\label{eq:PB}
-\eps\Delta\Phi^* = \rho^* = \sum_{i=1}^Nz_ic_i^*, \qquad \Phi^*=W \ \text{on } \Sigma_D, \qquad \partial_n\Phi^*=0 \ \text{on } \Sigma_N.
\end{equation}
Since $\nabla c_i^*=-z_ic_i^*\nabla\Phi^*$, we have $\rho^*\nabla\Phi^*=-\nabla\sum_ic_i^*$, so that $(u^*,c_1^*,\dots,c_N^*,\Phi^*)$ with pressure $p^*=K\sum_ic_i^*$ is a steady state of \eqref{eq:NP}--\eqref{eq:NS} satisfying \eqref{eq:bc}. Note that $c_i^*=\gamma_i$ on $\Gamma_i$ by \eqref{eq:compat}, and that $\ln c_i^*=\ln Z_i-z_i\Phi^*$. When blocking species are present, \eqref{eq:PB} is a nonlocal problem, since the constants $Z_i$ of the blocking species depend on $\Phi^*$. Throughout this paper, we assume that a Boltzmann state $\Phi^*\in H^1(\Omega)\cap L^\infty(\Omega)$ exists; when all species are in contact with a reservoir, existence and uniqueness follow from the strict monotonicity of the nonlinearity (see Appendix \ref{sec:elliptic}), and in general one may minimize the associated convex free energy under the mass constraints.

\subsection{The dissipative structure}

\begin{proposition}\label{prop:energy}
Let $(u,c_1,\dots,c_N,\Phi)$ be a global classical solution of \eqref{eq:NP}--\eqref{eq:NS} with boundary conditions \eqref{eq:bc} and positive concentrations. Then
\begin{multline}\label{eq:energy}
\frac{d}{dt}\Big[\Hst + \Pot + \Kin\Big]\\
+ \Diss + \frac{\nu}{K}\int_\Omega|\nabla u|^2dx = 0.
\end{multline}
\end{proposition}

\begin{proof}
We multiply \eqref{eq:NP} by
\[
\ln\frac{c_i}{c_i^*}+z_i(\Phi-\Phi^*) = \ln c_i+z_i\Phi-\ln Z_i,
\]
sum over all indices $i \in \left\{1, \dots, N \right\}$, and integrate over $\Omega$. Since $\nabla c_i+z_ic_i\nabla\Phi=c_i\nabla(\ln c_i+z_i\Phi)$, the diffusion terms give
\begin{align*}
&\sum_{i=1}^ND_i\int_\Omega\dv(\nabla c_i+z_ic_i\nabla\Phi)(\ln c_i+z_i\Phi-\ln Z_i)\,dx\\
&= -\Diss + \sum_{i=1}^ND_i\int_{\partial\Omega}(\partial_nc_i+z_ic_i\partial_n\Phi)(\ln c_i+z_i\Phi-\ln Z_i)\,dS.
\end{align*}
The boundary integral vanishes. Indeed, on $\partial\Omega\setminus\Gamma_i$ the first factor vanishes by \eqref{eq:bc}, and on $\Gamma_i\subset\Sigma_D$ we have $\ln c_i+z_i\Phi-\ln Z_i=\ln\gamma_i+z_iW-\ln Z_i=0$ by \eqref{eq:compat}.

On the left-hand side of \eqref{eq:NP}, we first consider the time derivatives. Since $c_i^*$ does not depend on time, $\partial_t(c_i\ln(c_i/c_i^*)-c_i+c_i^*)=\partial_tc_i\ln(c_i/c_i^*)$, and therefore
\[
\sum_{i=1}^N\int_\Omega\partial_tc_i\Big(\ln\frac{c_i}{c_i^*}+z_i(\Phi-\Phi^*)\Big)dx = \frac{d}{dt}\Hst + \int_\Omega(\partial_t\rho)(\Phi-\Phi^*)\,dx.
\]
Using \eqref{eq:Poisson} and integrating by parts, we have
\begin{align*}
\int_\Omega(\partial_t\rho)(\Phi-\Phi^*)\,dx &= -\eps\int_\Omega(\partial_t\Delta\Phi)(\Phi-\Phi^*)\,dx\\
&= \eps\int_\Omega(\partial_t\nabla\Phi)\cdot\nabla(\Phi-\Phi^*)\,dx - \eps\int_{\partial\Omega}(\partial_t\partial_n\Phi)(\Phi-\Phi^*)\,dS\\
&= \frac{\eps}{2}\frac{d}{dt}\int_\Omega|\nabla(\Phi-\Phi^*)|^2dx,
\end{align*}
where the boundary integral vanishes because $\Phi-\Phi^*=0$ on $\Sigma_D$ and $\partial_t\partial_n\Phi=0$ on $\Sigma_N$, and where we used that $\Phi^*$ does not depend on time. Next we consider the transport terms. Since $\dv u=0$ and $u=0$ on $\partial\Omega$, the integrals $\int_\Omega u\cdot\nabla(c_i\ln c_i-c_i)\,dx$ and $\int_\Omega u\cdot\nabla c_i\,dx$ vanish, and
\begin{align*}
&\sum_{i=1}^N\int_\Omega(u\cdot\nabla c_i)(\ln c_i+z_i\Phi-\ln Z_i)\,dx \\&= \sum_{i=1}^N\int_\Omega u\cdot\nabla(c_i\ln c_i-c_i)\,dx + \int_\Omega(u\cdot\nabla\rho)\Phi\,dx
- \sum_{i=1}^N\ln Z_i\int_\Omega u\cdot\nabla c_i\,dx\\
&= -\int_\Omega\rho\,u\cdot\nabla\Phi\,dx.
\end{align*}
Putting these computations together, we obtain
\begin{equation}\label{eq:energyNP}
\begin{aligned}
&\frac{d}{dt}\Big[\Hst + \Pot\Big] \\&\quad\quad+ \Diss 
= \int_\Omega\rho\,u\cdot\nabla\Phi\,dx.
\end{aligned}
\end{equation}
On the other hand, taking the $L^2$ inner product of \eqref{eq:NS} with $u$, and using $\int_\Omega(u\cdot\nabla u)\cdot u\,dx=0$ and $\int_\Omega\nabla p\cdot u\,dx=0$, we have
\[
\frac12\frac{d}{dt}\int_\Omega|u|^2dx + \nu\int_\Omega|\nabla u|^2dx = -K\int_\Omega\rho\,u\cdot\nabla\Phi\,dx.
\]
Dividing this identity by $K$ and adding it to \eqref{eq:energyNP}, the coupling terms cancel, and we obtain \eqref{eq:energy}.
\end{proof}

At each time $t$, we compare the solution with the Boltzmann state generated by its own potential. We define
\begin{equation}\label{eq:inst}
\tilde c_i(t) = \tilde Z_i(t)\,e^{-z_i\Phi(t)}, \qquad \tilde Z_i = Z_i \ \text{ if } \Gamma_i\neq\emptyset, \qquad \tilde Z_i(t) = \frac{\int_\Omega c_i(0)\,dx}{\int_\Omega e^{-z_i\Phi(t)}dx} \ \text{ if } \Gamma_i=\emptyset.
\end{equation}
By \eqref{eq:mass}, a blocking species satisfies $\int_\Omega\tilde c_i\,dx=\int_\Omega c_i\,dx=\int_\Omega c_i^*\,dx=\int_\Omega c_i(0)\,dx$. If $\Gamma_i\neq\emptyset$, then on $\Gamma_i\subset\Sigma_D$ we have
\begin{equation}\label{eq:inst-bc}
c_ie^{z_i\Phi} = \gamma_ie^{z_iW} = Z_i.
\end{equation}
Moreover, letting $f_i=c_ie^{z_i\Phi}$, we have $\nabla\ln f_i=\nabla(\ln c_i+z_i\Phi)$, and therefore
\begin{equation}\label{eq:fisher}
c_i|\nabla(\ln c_i+z_i\Phi)|^2 = c_i\frac{|\nabla f_i|^2}{f_i^2} = e^{-z_i\Phi}\frac{|\nabla f_i|^2}{f_i} = 4\,e^{-z_i\Phi}\Big|\nabla\sqrt{c_i\,e^{z_i\Phi}}\Big|^2.
\end{equation}
Thus the dissipation is a weighted Fisher information, with weight $e^{-z_i\Phi(t)}$.

\begin{lemma}\label{lem:compare}
For all $t\ge0$,
\[
\sum_{i=1}^N\int_\Omega\Big(c_i\ln\frac{c_i}{\tilde c_i}-c_i+\tilde c_i\Big)dx \ge \Hst + \eps\int_\Omega|\nabla(\Phi-\Phi^*)|^2dx.
\]
\end{lemma}

\begin{proof}
Since $\ln(c_i^*/\tilde c_i)=\ln(Z_i/\tilde Z_i)+z_i(\Phi-\Phi^*)$, we have for each $i$
\begin{multline*}
\int_\Omega\Big(c_i\ln\frac{c_i}{\tilde c_i}-c_i+\tilde c_i\Big)dx - \int_\Omega\Big(c_i\ln\frac{c_i}{c_i^*}-c_i+c_i^*\Big)dx\\
= z_i\int_\Omega c_i(\Phi-\Phi^*)\,dx + \ln\frac{Z_i}{\tilde Z_i}\int_\Omega c_i\,dx + \int_\Omega(\tilde c_i-c_i^*)\,dx.
\end{multline*}
Suppose first that $\Gamma_i\neq\emptyset$. Then $\tilde Z_i=Z_i$, the second term vanishes, and $\tilde c_i-c_i^*=c_i^*(e^{-z_i(\Phi-\Phi^*)}-1)\ge-z_ic_i^*(\Phi-\Phi^*)$ by the elementary inequality $e^{-x}-1\ge-x$. Hence the right-hand side is bounded from below by $z_i\int_\Omega(c_i-c_i^*)(\Phi-\Phi^*)\,dx$.

Suppose now that $\Gamma_i=\emptyset$, and let $N_i=\int_\Omega c_i(0)\,dx$. Then $\int_\Omega c_i\,dx=\int_\Omega\tilde c_i\,dx=\int_\Omega c_i^*\,dx=N_i$, so the third term vanishes. Since $c_i^*\,dx/N_i$ is a probability measure and $\ln$ is concave, Jensen's inequality gives
\[
\ln\frac{Z_i}{\tilde Z_i} = \ln\Big(\frac{\int_\Omega e^{-z_i\Phi}dx}{\int_\Omega e^{-z_i\Phi^*}dx}\Big) = \ln\Big(\int_\Omega e^{-z_i(\Phi-\Phi^*)}\frac{c_i^*}{N_i}\,dx\Big) \ge -\frac{z_i}{N_i}\int_\Omega c_i^*(\Phi-\Phi^*)\,dx.
\]
Multiplying by $N_i$, we see that the right-hand side is again bounded from below by $z_i\int_\Omega(c_i-c_i^*)(\Phi-\Phi^*)\,dx$.

Summing in $i$, and using \eqref{eq:Poisson}, \eqref{eq:PB}, $\Phi-\Phi^*=0$ on $\Sigma_D$ and $\partial_n(\Phi-\Phi^*)=0$ on $\Sigma_N$, we obtain
\begin{align*}
\sum_{i=1}^Nz_i\int_\Omega(c_i-c_i^*)(\Phi-\Phi^*)\,dx &= \int_\Omega(\rho-\rho^*)(\Phi-\Phi^*)\,dx = -\eps\int_\Omega\Delta(\Phi-\Phi^*)(\Phi-\Phi^*)\,dx\\
&= \eps\int_\Omega|\nabla(\Phi-\Phi^*)|^2dx - \eps\int_{\partial\Omega}\partial_n(\Phi-\Phi^*)(\Phi-\Phi^*)\,dS\\
&= \eps\int_\Omega|\nabla(\Phi-\Phi^*)|^2dx,
\end{align*}
which ends the proof.
\end{proof}

\subsection{Exponential decay}

We first prove exponential decay under a pointwise bound on the potential. Since the system is autonomous, we state it from an arbitrary initial time $t_0$.

\begin{proposition}\label{prop:decayM}
Let $(u,c_1,\dots,c_N,\Phi)$ be a global classical solution of \eqref{eq:NP}--\eqref{eq:NS} with boundary conditions \eqref{eq:bc} and positive concentrations. Let $t_0\ge0$ be a nonnegative time. Suppose that
\begin{equation}\label{eq:Mbound}
\|\Phi(t)\|_{L^\infty} \le M 
\end{equation} for all $t \ge t_0$. 
Let $B>0$ be such that, for every $i$ with $\Gamma_i\neq\emptyset$,
\begin{equation}\label{eq:Bbound}
\begin{gathered}
\frac{1}{Z_i}\int_\Omega c_i^*\,dx \le B,\\
\frac{1}{Z_i}\Big[\HstAt{t_0} + \PotAt{t_0} + \KinAt{t_0}\Big] \le B.
\end{gathered}
\end{equation}
Then there exists a constant $r>0$ depending only on $d$, $\Omega$, $\Gamma_1,\dots,\Gamma_N$, $D_1,\dots,D_N$, $z_1,\dots,z_N$, $\nu$, $M$ and $B$, such that for all $t\ge t_0$,
\begin{multline}\label{eq:decayM}
\Hst + \Pot + \Kin\\
\le \Big[\HstAt{t_0} + \PotAt{t_0} + \KinAt{t_0}\Big]e^{-r(t-t_0)}.
\end{multline}
Moreover, for each $i \in \left\{1, \dots, N\right\}$ and all $t\ge t_0$,
\begin{multline}\label{eq:L1M}
\|c_i(t)-c_i^*\|_{L^1}^2\\
\le C\Big[\HstAt{t_0} + \PotAt{t_0} + \KinAt{t_0}\Big]e^{-r(t-t_0)},
\end{multline}
where $C$ depends only on upper bounds for $\int_\Omega c_i^*\,dx$ and for the quantity in brackets.
\end{proposition}

\begin{proof}
By Proposition \ref{prop:energy}, the energy
\[
\Hst + \Pot + \Kin
\]
is nonincreasing in time. We bound the masses first. By Lemma \ref{lem:elem}, $a\mapsto a\ln a-a+1$ is convex, so Jensen's inequality for the probability measure $c_i^*\,dx/\int_\Omega c_i^*\,dx$ gives
\[
\int_\Omega\Big(c_i\ln\frac{c_i}{c_i^*}-c_i+c_i^*\Big)dx = \int_\Omega c_i^*\Big(\frac{c_i}{c_i^*}\ln\frac{c_i}{c_i^*}-\frac{c_i}{c_i^*}+1\Big)dx \ge \int_\Omega c_i^*\,dx\,\big(a_i\ln a_i-a_i+1\big),
\]
where $a_i=\int_\Omega c_i\,dx/\int_\Omega c_i^*\,dx$. If $a_i(t)\ge e^2$, then \eqref{eq:elem3} shows that the right-hand side is at least $\int_\Omega c_i(t)\,dx$. Since the left-hand side is bounded by the value of the nonincreasing quantity above at time $t_0$, we conclude that, for all $t\ge t_0$,
\begin{multline}\label{eq:massbound}
\int_\Omega c_i(t)\,dx \le \max\Big\{e^2\int_\Omega c_i^*\,dx,\\
\HstAt{t_0} + \PotAt{t_0} + \KinAt{t_0}\Big\}.
\end{multline}
In particular, if $\Gamma_i\neq\emptyset$, then $\int_\Omega c_i(t)\,dx\le e^2Z_iB$ by \eqref{eq:Bbound}.

Next we fix $t\ge t_0$ and $i \in \left\{1, \dots, N \right\}$, and apply Theorem \ref{thm:wLSI} with $g=c_i(t)$, $V=z_i\Phi(t)$ and $w=e^{-z_i\Phi(t)}$. By \eqref{eq:Mbound}, $\osc V\le2|z_i|M$ and $\int_\Omega e^{-z_i\Phi}dx\ge|\Omega|e^{-|z_i|M}$. If $\Gamma_i=\emptyset$, the left-hand side of \eqref{eq:wLSI} is
\[
\int_\Omega c_i\ln\Big(c_ie^{z_i\Phi}\frac{\int_\Omega e^{-z_i\Phi}dx}{\int_\Omega c_i\,dx}\Big)dx = \int_\Omega c_i\ln\frac{c_i}{\tilde c_i}\,dx = \int_\Omega\Big(c_i\ln\frac{c_i}{\tilde c_i}-c_i+\tilde c_i\Big)dx,
\]
because $c_i$ and $\tilde c_i$ have the same mass. If $\Gamma_i\neq\emptyset$, then ${c_i}/w=Z_i$ on $\Gamma_i$ by \eqref{eq:inst-bc}. The left-hand side of \eqref{eq:wdirLSI} with $g= c_i$, $w = e^{-z_i\Phi}$ and $Z=Z_i$ is $\int_\Omega(c_i\ln(c_i/\tilde c_i)-c_i+\tilde c_i)\,dx$, and
\[
\frac{\int_\Omega c_i\,dx}{Z_i\int_\Omega e^{-z_i\Phi}dx} \le \frac{e^2Z_iB}{Z_i|\Omega|e^{-|z_i|M}} = \frac{e^{2+|z_i|M}B}{|\Omega|}.
\]
In both cases, \eqref{eq:fisher} and Theorem \ref{thm:wLSI} yield
\[
D_i\int_\Omega c_i|\nabla(\ln c_i+z_i\Phi)|^2dx = 4D_i\int_\Omega e^{-z_i\Phi}\Big|\nabla\sqrt{c_ie^{z_i\Phi}}\Big|^2dx \ge \frac{4D_i}{\Lambda_i}\int_\Omega\Big(c_i\ln\frac{c_i}{\tilde c_i}-c_i+\tilde c_i\Big)dx,
\]
where $\Lambda_i=C_0e^{2|z_i|M}$ if $\Gamma_i=\emptyset$ and $\Lambda_i=C_0e^{2|z_i|M}(1+e^{2+|z_i|M}B/|\Omega|)$ if $\Gamma_i\neq\emptyset$, and $C_0$ is the maximum of the constants of Theorem \ref{thm:wLSI} with $\Gamma=\Gamma_i$, $i=1,\dots,N$. Summing in $i$ and using Lemma \ref{lem:compare}, we obtain
\begin{align*}
\Diss &\ge \min_{1\le i\le N}\frac{4D_i}{\Lambda_i}\sum_{i=1}^N\int_\Omega\Big(c_i\ln\frac{c_i}{\tilde c_i}-c_i+\tilde c_i\Big)dx\\
&\ge \min_{1\le i\le N}\frac{4D_i}{\Lambda_i}\Big[\Hst + \Pot\Big],
\end{align*}
where in the last step we only used half of the term $\eps\int_\Omega|\nabla(\Phi-\Phi^*)|^2dx$ provided by Lemma \ref{lem:compare}. Moreover, by the Poincar\'e inequality $\int_\Omega|u|^2dx\le C_1\int_\Omega|\nabla u|^2dx$, where $C_1$ depends only on $\Omega$,
\[
\frac{\nu}{K}\int_\Omega|\nabla u|^2dx \ge \frac{2\nu}{C_1}\cdot\frac{1}{2K}\int_\Omega|u|^2dx.
\]
We let $r=\min\{\min_i4D_i/\Lambda_i,\,2\nu/C_1\}$. Inserting these bounds in \eqref{eq:energy}, we infer that
\begin{multline*}
\frac{d}{dt}\Big[\Hst + \Pot + \Kin\Big]\\
+ r\Big[\Hst + \Pot + \Kin\Big] \le 0
\end{multline*}
for $t\ge t_0$, and \eqref{eq:decayM} follows from Gr\"onwall's inequality.

Finally, by \eqref{eq:elem2} with $a=c_i/c_i^*$ and the Cauchy--Schwarz inequality with respect to the measure $c_i^*\,dx$,
\begin{align*}
\|c_i-c_i^*\|_{L^1} = \int_\Omega c_i^*\Big|\frac{c_i}{c_i^*}-1\Big|dx &\le \Big(\int_\Omega c_i^*\,\frac{2c_i/c_i^*+4}{3}\,dx\Big)^{1/2}\Big(\int_\Omega\Big(c_i\ln\frac{c_i}{c_i^*}-c_i+c_i^*\Big)dx\Big)^{1/2}\\
&= \Big(\frac23\big(\|c_i\|_{L^1}+2\|c_i^*\|_{L^1}\big)\Big)^{1/2}\Big(\int_\Omega\Big(c_i\ln\frac{c_i}{c_i^*}-c_i+c_i^*\Big)dx\Big)^{1/2},
\end{align*}
which, together with \eqref{eq:massbound} and \eqref{eq:decayM}, gives \eqref{eq:L1M}.
\end{proof}

The rate $r$ in Proposition \ref{prop:decayM} depends on $M$, and all known uniform-in-time bounds on $\|\Phi(t)\|_{L^\infty}$ depend on negative powers of $\eps$ (see \cite[Remark 2]{abdo2025logarithmic}). We now show that the asymptotic rate is nevertheless independent of $\eps$. For this purpose we assume:
\begin{enumerate}
\item[(U)] at least one species with positive valence and at least one species with negative valence are in contact with a reservoir: there exist $i,j\in\{1,\dots,N\}$ with $\Gamma_i\neq\emptyset$, $\Gamma_j\neq\emptyset$, $z_i>0$ and $z_j<0$.
\end{enumerate}
The remaining species may be blocking. Under (U), the constants $Z_i$ of the species in contact with a reservoir are determined by the boundary data through \eqref{eq:compat}, and the masses $N_i=\int_\Omega c_i(0)\,dx$ of the blocking species are conserved by \eqref{eq:mass}; none of these quantities depends on $\eps$. The function $F(\phi)=\sum_iz_iZ_ie^{-z_i\phi}$, with $Z_i$ given by \eqref{eq:Boltz}, is strictly decreasing and has a unique zero $\Phi_b$ (see Appendix \ref{sec:elliptic}), and $\Phi^*$ satisfies \eqref{eq:weakPB}. We let $M_*$ be given by \eqref{eq:Mstar} if $\Gamma_i\neq\emptyset$ for all $i$, and by \eqref{eq:Mstarblock} otherwise. Then Proposition \ref{prop:PB} (if all species are in contact with a reservoir) or Proposition \ref{prop:PBblock} (otherwise) gives
\begin{equation}\label{eq:cstarbounds}
\|\Phi^*\|_{L^\infty} \le M_*, \qquad Z_ie^{-|z_i|M_*} \le c_i^* \le Z_ie^{|z_i|M_*}, \qquad Z_i\le\bar Z_i,
\end{equation}
where $\bar Z_i=Z_i$ if $\Gamma_i\neq\emptyset$ and $\bar Z_i=N_ie^{|z_i|M_*}/|\Omega|$ if $\Gamma_i=\emptyset$. The constants $M_*$ and $\bar Z_i$ are independent of $\eps$.

\begin{theorem}\label{thm:main}
Let $(u,c_1,\dots,c_N,\Phi)$ be a global classical solution of \eqref{eq:NP}--\eqref{eq:NS} with boundary conditions \eqref{eq:bc} and positive concentrations. Suppose that the condition (U) holds, and that for some $p>d/2$,
\begin{equation}\label{eq:Lp}
R := \sup_{t\ge0}\ \max_{1\le i\le N}\|c_i(t)\|_{L^p} < \infty,
\end{equation}
where $R$ is allowed to depend on $\eps$. Then there exist a constant $r_*>0$ depending only on $d$, $\Omega$, $\Gamma_1,\dots,\Gamma_N$, $D_1,\dots,D_N$, $z_1,\dots,z_N$, $\nu$ and $M_*$, and a time $T_\eps\ge0$, such that for all $t\ge T_\eps$,
\begin{multline}\label{eq:main}
\Hst + \Pot + \Kin\\
\le \Big[\HstAt{T_\eps} + \PotAt{T_\eps} + \KinAt{T_\eps}\Big]e^{-r_*(t-T_\eps)},
\end{multline}
and, for each $i \in \left\{1, \dots, N \right\}$,
\begin{multline}\label{eq:mainL1}
\|c_i(t)-c_i^*\|_{L^1}^2\\
\le C\Big[\HstAt{T_\eps} + \PotAt{T_\eps} + \KinAt{T_\eps}\Big]e^{-r_*(t-T_\eps)},
\end{multline}
where $C$ depends only on $\Omega$, $z_i$, $\bar Z_i$ and $M_*$. 
\end{theorem}

\begin{proof} The proof is divided into several steps.

{\bf{Step 1: a uniform bound on the potential.}} Let $\psi=\Phi-\Phi^*$. By \eqref{eq:Poisson} and \eqref{eq:PB}, for every nonnegative time $t$, the function $\psi(t)$ belongs to $H^1_{\Sigma_D}(\Omega)$ and satisfies \eqref{eq:weakmixed} with $f=\rho(t)-\rho^*$. By \eqref{eq:Lp} and \eqref{eq:cstarbounds},
\[
\|\rho(t)-\rho^*\|_{L^p} \le \sum_{i=1}^N|z_i|\big(\|c_i(t)\|_{L^p}+\|c_i^*\|_{L^p}\big) \le \sum_{i=1}^N|z_i|\big(R+\bar Z_ie^{|z_i|M_*}|\Omega|^{1/p}\big) =: R'.
\]
Proposition \ref{prop:Linfty} with $q=p$ and \eqref{eq:cstarbounds} yield
\[
\|\Phi(t)\|_{L^\infty} \le M_* + \frac{C}{\eps}R' 
\] for all $t \ge 0$, where $C$ is a positive constant depending only on $\Omega$, $\Sigma_D$, $d$ and $p$.

{\bf{Step 2: qualitative convergence.}} By Step 1, Proposition \ref{prop:decayM} applies with $t_0=0$, $M=M_*+CR'/\eps$ and a suitable $B$. Hence there exists $r_\eps>0$ such that
\begin{multline*}
\Hst + \Pot + \Kin\\
\le \Big[\HstAt{0} + \PotAt{0} + \KinAt{0}\Big]e^{-r_\eps t},
\end{multline*} for any $t \ge 0$,
and, by \eqref{eq:L1M}, $\|c_i(t)-c_i^*\|_{L^1}\to0$ as $t\to\infty$ for each $i \in \left\{1, \dots, N \right\}$.

{\bf{Step 3: convergence of the potential in $L^\infty$.}} Fix $q\in(d/2,p)$ and let $\theta\in(0,1)$ be defined by $1/q=\theta+(1-\theta)/p$. By H\"older's inequality,
\[
\|\rho(t)-\rho^*\|_{L^q} \le \|\rho(t)-\rho^*\|_{L^1}^{\theta}\|\rho(t)-\rho^*\|_{L^p}^{1-\theta} \le \Big(\sum_{i=1}^N|z_i|\|c_i(t)-c_i^*\|_{L^1}\Big)^{\theta}(R')^{1-\theta},
\]
which tends to $0$ as $t\to\infty$ by Step 2. By Proposition \ref{prop:Linfty}, there is a constant $C$ depending only on $\Omega$, $\Sigma_D$, $d$ and $q$ such that
\[
\|\Phi(t)-\Phi^*\|_{L^\infty} \le \frac{C}{\eps}\|\rho(t)-\rho^*\|_{L^q} \longrightarrow 0
\] as $t\to\infty$.

{\bf{Step 4: choice of $T_\eps$.}} By Steps 2 and 3, there exists $T_\eps\ge0$ such that
\[
\|\Phi(t)-\Phi^*\|_{L^\infty} \le 1 
\] for all $t\ge T_\eps$, 
and
\begin{multline*}
\HstAt{T_\eps} + \PotAt{T_\eps}\\
+ \KinAt{T_\eps} \le e^2\min_{1\le j\le N}\int_\Omega c_j^*\,dx.
\end{multline*}

{\bf{Step 5: $\eps$-uniform constants.}} We apply Proposition \ref{prop:decayM} with $t_0=T_\eps$ and $M=M_*+1$; condition \eqref{eq:Mbound} holds by Step 4 and \eqref{eq:cstarbounds}. The condition \eqref{eq:Bbound} only involves the species with $\Gamma_i\neq\emptyset$. By Step 4 and \eqref{eq:cstarbounds}, for every such $i$,
\[
\frac{1}{Z_i}\int_\Omega c_i^*\,dx \le |\Omega|e^{|z_i|M_*},
\]
\begin{multline*}
\frac{1}{Z_i}\Big[\HstAt{T_\eps} + \PotAt{T_\eps} + \KinAt{T_\eps}\Big]\\
\le \frac{e^2}{Z_i}\int_\Omega c_i^*\,dx \le e^2|\Omega|e^{|z_i|M_*},
\end{multline*}
so that \eqref{eq:Bbound} holds with $B=e^2|\Omega|\max_ie^{|z_i|M_*}$, which depends only on $\Omega$, $z_1,\dots,z_N$ and $M_*$. Hence the rate $r$ given by Proposition \ref{prop:decayM} depends only on $d$, $\Omega$, $\Gamma_1,\dots,\Gamma_N$, $D_1,\dots,D_N$, $z_1,\dots,z_N$, $\nu$ and $M_*$; we denote it by $r_*$. Then \eqref{eq:decayM} gives \eqref{eq:main}. Moreover, for every $i$, $\int_\Omega c_i^*\,dx\le \bar Z_i|\Omega|e^{|z_i|M_*}$ by \eqref{eq:cstarbounds}, and the quantity in brackets in \eqref{eq:main} is bounded by $e^2\int_\Omega c_i^*\,dx\le e^2\bar Z_i|\Omega|e^{|z_i|M_*}$, so \eqref{eq:L1M} gives \eqref{eq:mainL1}. 
\end{proof}

\begin{remark}
When the boundary data are constant and electroneutral, that is, $W=0$, $\Gamma_i=\Sigma_D$ and $\gamma_i$ constant with $\sum_iz_i\gamma_i=0$, the condition (U) holds as soon as some $z_i\neq0$, and we have $Z_i=\gamma_i$, $F(0)=\sum_iz_i\gamma_i=0$, so $\Phi_b=0$ and $M_*=0$. Then $\Phi^*\equiv0$ and $c_i^*\equiv\gamma_i$, and Theorem \ref{thm:main} gives an asymptotic rate $r_*$ depending only on $d$, $\Omega$, $\Sigma_D$, $D_1,\dots,D_N$, $z_1,\dots,z_N$ and $\nu$, and in particular not on $\eps$ or on the boundary concentrations $\gamma_i$. This is the analogue, for reservoir boundary conditions, of the $\eps$-independent rate obtained in \cite[Theorem 2]{abdo2025logarithmic} for homogeneous Neumann boundary conditions.
\end{remark}

\begin{remark}
The rate $r_*$ in Theorem \ref{thm:main} is independent of $\eps$, but the time $T_\eps$ after which it is attained is not quantified. Steps 2--4 only provide $T_\eps$ in terms of the rate $r_\eps$ of Step 2, which depends on $M=M_*+CR'/\eps$ through the factors $e^{2|z_i|M}$ in Proposition \ref{prop:decayM}, and may therefore deteriorate as $\eps\to0$. Theorem \ref{thm:main} is thus a statement about the asymptotic rate of convergence.
\end{remark}

\begin{remark}
Assumption (U) is used only to ensure that $M_*$ and the constants $\bar Z_i$ do not depend on $\eps$. Proposition \ref{prop:decayM} holds without it. If all species are blocking, or if all species in contact with a reservoir have valences of the same sign, the argument of Proposition \ref{prop:PBblock} only yields an inequality between the masses of the species and no bound on $\Phi_b$, and Theorem \ref{thm:main} remains valid, with the same proof, provided one has bounds on $\|\Phi^*\|_{L^\infty}$ that are uniform in $\eps$. Assumption (U) is natural: it holds whenever the system is in contact with an electrolyte reservoir.
\end{remark}

\begin{remark}
The compatibility condition \eqref{eq:compat} is necessary for the existence of a Boltzmann steady state satisfying \eqref{eq:bc}. Without it, the boundary integral in the proof of Proposition \ref{prop:energy} does not vanish, the steady states carry nonzero electric currents, and the long-time behavior is of a different nature (see \cite{constantin2021nernst}). Finally, the restriction $d\in\{2,3,4\}$ comes only from the logarithmic Sobolev inequality \eqref{eq:AL} established in \cite{abdo2025logarithmic}.
\end{remark}

\appendix

\section{Elliptic Estimates}\label{sec:elliptic}

In this appendix we prove the elliptic estimates used in Section \ref{sec:app}. Throughout, $\Sigma_D$, $\Sigma_N$ and $H^1_{\Sigma_D}(\Omega)$ are as in Section \ref{sec:app}. We fix an exponent
\begin{equation}\label{eq:s}
s = \frac{2d}{d-2} \ \text{ if } d=3,4, \qquad s<\infty \text{ arbitrary if } d=2.
\end{equation}

\begin{lemma}\label{lem:sob}
There exists a constant $C_0>0$ depending only on $\Omega$, $\Sigma_D$, $d$ and $s$, such that
\[
\|v\|_{L^s} \le C_0\|\nabla v\|_{L^2} 
\] for all $v\in H^1_{\Sigma_D}(\Omega).$
\end{lemma}

\begin{proof} By the Sobolev embedding $H^1(\Omega)\hookrightarrow L^s(\Omega)$, we have $\|v\|_{L^s}\le C(\|v\|_{L^2}+\|\nabla v\|_{L^2})$, and the Friedrichs inequality \eqref{eq:friedrichs} with $\Gamma=\Sigma_D$ bounds $\|v\|_{L^2}$ by $C\|\nabla v\|_{L^2}$.
\end{proof}

\begin{lemma}\label{lem:iter}
Let $k_0\in\mathbb R$, and let $\phi:[k_0,\infty)\to[0,\infty)$ be nonincreasing. Suppose that there exist constants $C>0$, $\alpha>0$ and $\beta>1$ such that
\[
\phi(h) \le \Big(\frac{C}{h-k}\Big)^{\alpha}\phi(k)^{\beta} 
\] for all $ h>k\ge k_0.$
Then $\phi(k_0+\delta)=0$, where $\delta=C\,\phi(k_0)^{(\beta-1)/\alpha}\,2^{\beta/(\beta-1)}$.
\end{lemma}

\begin{proof}
If $\phi(k_0)=0$ there is nothing to prove, so we assume $\phi(k_0)>0$. Let $\mu=\alpha/(\beta-1)$ and $k_n=k_0+\delta(1-2^{-n})$ for $n\ge0$. We claim that
\[
\phi(k_n) \le \phi(k_0)\,2^{-n\mu} 
\] for all $n\ge0.$
This is clear for $n=0$. Suppose it holds for some $n$. Since $k_{n+1}-k_n=\delta2^{-(n+1)}$, the assumption gives
\[
\phi(k_{n+1}) \le \Big(\frac{C2^{n+1}}{\delta}\Big)^{\alpha}\phi(k_n)^{\beta} \le C^\alpha\delta^{-\alpha}2^{(n+1)\alpha}\phi(k_0)^{\beta}2^{-n\mu\beta}.
\]
By the choice of $\delta$, we have $\delta^\alpha=C^\alpha\phi(k_0)^{\beta-1}2^{\alpha\beta/(\beta-1)}=C^\alpha\phi(k_0)^{\beta-1}2^{\alpha+\mu}$, so that
\[
\phi(k_{n+1}) \le \phi(k_0)\,2^{(n+1)\alpha-\alpha-\mu-n\mu\beta} = \phi(k_0)\,2^{n(\alpha-\mu\beta)-\mu} = \phi(k_0)\,2^{-(n+1)\mu},
\]
where we used $\alpha=\mu(\beta-1)$. This proves the claim. Since $\phi$ is nonincreasing and $k_n<k_0+\delta$, we deduce that $\phi(k_0+\delta)\le\phi(k_n)\le\phi(k_0)2^{-n\mu}$ for all $n$, hence $\phi(k_0+\delta)=0$.
\end{proof}

\begin{proposition}\label{prop:Linfty}
Let $q>d/2$ and $f\in L^q(\Omega)$. Let $\psi\in H^1_{\Sigma_D}(\Omega)$ satisfy
\begin{equation}\label{eq:weakmixed}
\eps\int_\Omega\nabla\psi\cdot\nabla\varphi\,dx = \int_\Omega f\varphi\,dx \qquad \text{for all } \varphi\in H^1_{\Sigma_D}(\Omega),
\end{equation}
that is, $\psi$ is a weak solution of $-\eps\Delta\psi=f$ in $\Omega$, $\psi=0$ on $\Sigma_D$, $\partial_n\psi=0$ on $\Sigma_N$. Then
\[
\|\psi\|_{L^\infty} \le \frac{C}{\eps}\|f\|_{L^q},
\]
where $C>0$ depends only on $\Omega$, $\Sigma_D$, $d$ and $q$.
\end{proposition}

\begin{proof}
Let $C_0$ be the constant of Lemma \ref{lem:sob}. We choose $s$ as in \eqref{eq:s}, with $s=4q/(q-1)$ if $d=2$, and set
\[
\beta = s\Big(1-\frac1q\Big)-1.
\]
If $d=3,4$, then $s(1-1/q)>s(1-2/d)=2$ because $q>d/2$; if $d=2$, then $s(1-1/q)=4$. In both cases $\beta>1$ and $1-1/q-1/s>0$.

For $k\ge0$, let $A(k)=\{x\in\Omega:\psi(x)>k\}$. Since $\psi=0$ on $\Sigma_D$ and $k\ge0$, the function $(\psi-k)_+$ belongs to $H^1_{\Sigma_D}(\Omega)$, and $\nabla(\psi-k)_+=\mathbf 1_{A(k)}\nabla\psi$. Taking $\varphi=(\psi-k)_+$ in \eqref{eq:weakmixed} and using H\"older's inequality with exponents $q$, $s$ and $(1-1/q-1/s)^{-1}$, we obtain
\[
\eps\int_\Omega|\nabla(\psi-k)_+|^2dx = \int_{A(k)}f(\psi-k)_+\,dx \le \|f\|_{L^q}\|(\psi-k)_+\|_{L^s}|A(k)|^{1-\frac1q-\frac1s}.
\]
By Lemma \ref{lem:sob}, $\|(\psi-k)_+\|_{L^s}^2\le C_0^2\|\nabla(\psi-k)_+\|_{L^2}^2$, and therefore
\[
\|(\psi-k)_+\|_{L^s} \le \frac{C_0^2}{\eps}\|f\|_{L^q}|A(k)|^{1-\frac1q-\frac1s}.
\]
For $h>k\ge0$ we have $(\psi-k)_+\ge h-k$ on $A(h)\subset A(k)$, so that
\[
(h-k)|A(h)|^{\frac1s} \le \|(\psi-k)_+\|_{L^s} \le \frac{C_0^2}{\eps}\|f\|_{L^q}|A(k)|^{1-\frac1q-\frac1s}.
\]
Raising to the power $s$, we arrive at
\[
|A(h)| \le \Big(\frac{C_0^2\|f\|_{L^q}}{\eps(h-k)}\Big)^{s}|A(k)|^{\beta}.
\]
The function $k\mapsto|A(k)|$ is nonincreasing, and $|A(0)|\le|\Omega|$. Lemma \ref{lem:iter} with $k_0=0$, $\alpha=s$ and $C=C_0^2\|f\|_{L^q}/\eps$ yields $|A(\delta)|=0$ with $\delta=\frac{C_0^2}{\eps}\|f\|_{L^q}|\Omega|^{(\beta-1)/s}2^{\beta/(\beta-1)}$, that is, $\psi\le\delta$ almost everywhere. Since $-\psi$ satisfies \eqref{eq:weakmixed} with $f$ replaced by $-f$, the same argument gives $-\psi\le\delta$. This proves the proposition with $C=C_0^2|\Omega|^{(\beta-1)/s}2^{\beta/(\beta-1)}$, which depends only on $\Omega$, $\Sigma_D$, $d$ and $q$.
\end{proof}

We now turn to the Poisson--Boltzmann equation. Let $Z_1,\dots,Z_N>0$ be constants and let $z_1,\dots,z_N$ be real numbers such that $z_i>0$ for some $i$ and $z_j<0$ for some $j$. We define
\[
F(\phi) = \sum_{i=1}^N z_iZ_i\,e^{-z_i\phi}, \qquad \phi\in\mathbb R.
\]
Then $F'(\phi)=-\sum_iz_i^2Z_ie^{-z_i\phi}<0$. Moreover, as $\phi\to-\infty$ the terms with $z_i>0$ tend to $+\infty$ while the others tend to $0$ or are constant, so $F(\phi)\to+\infty$; similarly $F(\phi)\to-\infty$ as $\phi\to+\infty$. Hence $F$ has a unique zero, which we denote by $\Phi_b$.

\begin{proposition}\label{prop:PB}
Let $W\in L^\infty(\Sigma_D)$ be the trace on $\Sigma_D$ of a function in $H^1(\Omega)$, and let $\Phi^*\in H^1(\Omega)\cap L^\infty(\Omega)$ satisfy $\Phi^*=W$ on $\Sigma_D$ and
\begin{equation}\label{eq:weakPB}
\eps\int_\Omega\nabla\Phi^*\cdot\nabla\varphi\,dx = \int_\Omega F(\Phi^*)\varphi\,dx \qquad \text{for all } \varphi\in H^1_{\Sigma_D}(\Omega),
\end{equation}
that is, $\Phi^*$ is a weak solution of $-\eps\Delta\Phi^*=F(\Phi^*)$ in $\Omega$, $\Phi^*=W$ on $\Sigma_D$, $\partial_n\Phi^*=0$ on $\Sigma_N$. Then
\[
\min\Big\{\inf_{\Sigma_D}W,\ \Phi_b\Big\} \le \Phi^* \le \max\Big\{\sup_{\Sigma_D}W,\ \Phi_b\Big\} \quad \text{a.e. in } \Omega.
\]
In particular, $\|\Phi^*\|_{L^\infty}\le M_*$, where
\begin{equation}\label{eq:Mstar}
M_* = \max\big\{\|W\|_{L^\infty(\Sigma_D)},\ |\Phi_b|\big\}
\end{equation}
depends only on $W$, $z_1,\dots,z_N$ and $Z_1,\dots,Z_N$, and not on $\eps$.
\end{proposition}

\begin{proof}
Let $m=\max\{\sup_{\Sigma_D}W,\Phi_b\}$. Since $\Phi^*=W\le m$ on $\Sigma_D$, the function $(\Phi^*-m)_+$ belongs to $H^1_{\Sigma_D}(\Omega)$, and $\nabla(\Phi^*-m)_+=\mathbf 1_{\{\Phi^*>m\}}\nabla\Phi^*$. Taking $\varphi=(\Phi^*-m)_+$ in \eqref{eq:weakPB}, we obtain
\[
\eps\int_\Omega|\nabla(\Phi^*-m)_+|^2dx = \int_{\{\Phi^*>m\}}F(\Phi^*)(\Phi^*-m)\,dx.
\]
Since $F$ is decreasing and $m\ge\Phi_b$, we have $F(\Phi^*)<F(\Phi_b)=0$ on $\{\Phi^*>m\}$, so the right-hand side is nonpositive. Hence $\nabla(\Phi^*-m)_+=0$, so $(\Phi^*-m)_+$ is constant in the connected domain $\Omega$. Since its trace vanishes on $\Sigma_D$, which has positive surface measure, this constant is zero, that is, $\Phi^*\le m$. The lower bound follows in the same way by taking $\varphi=(m'-\Phi^*)_+$ with $m'=\min\{\inf_{\Sigma_D}W,\Phi_b\}$, and using $F(\Phi^*)>0$ on $\{\Phi^*<m'\}$.
\end{proof}

When blocking species are present, the constants $Z_k$ of these species in \eqref{eq:Boltz} depend on $\Phi^*$, and so does $\Phi_b$. The following proposition shows that $\Phi_b$, and hence $\Phi^*$, is nevertheless bounded independently of $\eps$, provided that reservoir species of both signs are present. We denote by $P_R$ and $N_R$ the sets of indices $i$ with $\Gamma_i\neq\emptyset$ and $z_i>0$, respectively $z_i<0$, and by $P_B$ and $N_B$ the sets of indices $k$ with $\Gamma_k=\emptyset$ and $z_k>0$, respectively $z_k<0$. Species with zero valence do not contribute to $F$. For $k\in P_B\cup N_B$, we let $N_k>0$ be given, and we define
\begin{equation}\label{eq:Apm}
A_+ = \sum_{i\in P_R}z_iZ_ie^{z_i\|W\|_{L^\infty(\Sigma_D)}} + \sum_{k\in P_B}\frac{z_kN_k}{|\Omega|}, \qquad A_- = \sum_{j\in N_R}|z_j|Z_je^{|z_j|\|W\|_{L^\infty(\Sigma_D)}} + \sum_{k\in N_B}\frac{|z_k|N_k}{|\Omega|}.
\end{equation}

\begin{proposition}\label{prop:PBblock}
Suppose that $P_R\neq\emptyset$ and $N_R\neq\emptyset$, and let $Z_i>0$ be given constants for $i\in P_R\cup N_R$. Let $W$ be as in Proposition \ref{prop:PB}, and let $\Phi^*\in H^1(\Omega)\cap L^\infty(\Omega)$ satisfy $\Phi^*=W$ on $\Sigma_D$ and \eqref{eq:weakPB}, where, for $k\in P_B\cup N_B$,
\[
Z_k = \frac{N_k}{\int_\Omega e^{-z_k\Phi^*}dx}.
\]
Then $\|\Phi^*\|_{L^\infty}\le M_*$, where
\begin{equation}\label{eq:Mstarblock}
M_* = \max\Big\{\|W\|_{L^\infty(\Sigma_D)},\ \min_{j\in N_R}\frac{1}{|z_j|}\ln\frac{A_+}{|z_j|Z_j},\ \min_{i\in P_R}\frac{1}{z_i}\ln\frac{A_-}{z_iZ_i}\Big\}
\end{equation}
depends only on $W$, $|\Omega|$, $z_1,\dots,z_N$, the constants $Z_i$ with $\Gamma_i\neq\emptyset$, and the masses $N_k$, and not on $\eps$. Moreover, for every $k\in P_B\cup N_B$,
\[
\frac{N_k}{|\Omega|}e^{-|z_k|M_*} \le Z_k \le \frac{N_k}{|\Omega|}e^{|z_k|M_*}.
\]
\end{proposition}

\begin{proof}
Once $\Phi^*$ is given, all the constants $Z_1,\dots,Z_N$ are fixed positive numbers, so Proposition \ref{prop:PB} applies, and it suffices to bound $\Phi_b$ in the case where it is the active bound.

We first prove the upper bound. If $\Phi_b\le\sup_{\Sigma_D}W$, then $\Phi^*\le\|W\|_{L^\infty(\Sigma_D)}$ by Proposition \ref{prop:PB}. Otherwise, $\Phi_b>\sup_{\Sigma_D}W\ge-\|W\|_{L^\infty(\Sigma_D)}$ and $\Phi^*\le\Phi_b$ in $\Omega$. For $i\in P_R$, we have $z_iZ_ie^{-z_i\Phi_b}\le z_iZ_ie^{z_i\|W\|_{L^\infty(\Sigma_D)}}$. For $k\in P_B$, since $\Phi^*\le\Phi_b$, we have $\int_\Omega e^{-z_k\Phi^*}dx\ge|\Omega|e^{-z_k\Phi_b}$, and therefore
\[
z_kZ_ke^{-z_k\Phi_b} = \frac{z_kN_ke^{-z_k\Phi_b}}{\int_\Omega e^{-z_k\Phi^*}dx} \le \frac{z_kN_k}{|\Omega|}.
\]
Hence the sum of the positive terms of $F(\Phi_b)$ is bounded by $A_+$. Since $F(\Phi_b)=0$, the sum of the absolute values of the negative terms is bounded by $A_+$ as well. Keeping a single term $j\in N_R$, we obtain $|z_j|Z_je^{|z_j|\Phi_b}\le A_+$, that is,
\[
\Phi_b \le \frac{1}{|z_j|}\ln\frac{A_+}{|z_j|Z_j}.
\]
Taking the minimum over $j\in N_R$ yields $\Phi^*\le M_*$.

The lower bound is obtained in the same way. If $\Phi_b\ge\inf_{\Sigma_D}W$, then $\Phi^*\ge-\|W\|_{L^\infty(\Sigma_D)}$. Otherwise, $\Phi_b<\inf_{\Sigma_D}W\le\|W\|_{L^\infty(\Sigma_D)}$ and $\Phi^*\ge\Phi_b$ in $\Omega$. For $j\in N_R$, we have $|z_j|Z_je^{|z_j|\Phi_b}\le|z_j|Z_je^{|z_j|\|W\|_{L^\infty(\Sigma_D)}}$, and for $k\in N_B$, since $\int_\Omega e^{|z_k|\Phi^*}dx\ge|\Omega|e^{|z_k|\Phi_b}$, we have $|z_k|Z_ke^{|z_k|\Phi_b}\le|z_k|N_k/|\Omega|$. Hence the sum of the absolute values of the negative terms of $F(\Phi_b)$ is bounded by $A_-$, and so is the sum of the positive terms. Keeping a single term $i\in P_R$, we obtain $z_iZ_ie^{-z_i\Phi_b}\le A_-$, that is,
\[
\Phi_b \ge -\frac{1}{z_i}\ln\frac{A_-}{z_iZ_i}.
\]
Taking the minimum over $i\in P_R$ yields $\Phi^*\ge-M_*$.

Finally, the bounds on $Z_k$ follow from $\|\Phi^*\|_{L^\infty}\le M_*$ and the definition of $Z_k$.
\end{proof}

\vspace{0.5cm}

{\bf{Acknowledgments.}} E.A. was partially supported by the University Research Board (URB) of the American University of Beirut under Grant No. 104752. 

\vspace{0.5cm}

{\bf{Data Availability Statement.}} The research does not have any associated data.

\vspace{0.5cm}

{\bf{Conflict of Interest.}} The author declares that he has no conflict of interest.
\vspace{0.5cm}

{\bf{Declaration on the use of generative AI.}} The problem and main ideas of the paper were suggested by the author based on his previous work \cite{abdo2025logarithmic}. A draft was generated by Opus 5.5 (Anthropic) and subsequently improved, partially rewritten, and thoroughly checked by the author. The author takes full responsibility for the accuracy and originality of the work.

\bibliographystyle{plain}
\bibliography{references}

@article{abdo2025logarithmic,
  title={Logarithmic Sobolev inequalities for bounded domains and applications to drift-diffusion equations},
  author={Abdo, Elie and Lee, Fizay-Noah},
  journal={Journal of Functional Analysis},
  volume={288},
  number={1},
  pages={110716},
  year={2025},
  publisher={Elsevier}
}

@article{constantin2019nernst,
  title={On the nernst--planck--navier--stokes system},
  author={Constantin, Peter and Ignatova, Mihaela},
  journal={Archive for Rational Mechanics and Analysis},
  volume={232},
  number={3},
  pages={1379--1428},
  year={2019},
  publisher={Springer}
}

@article{constantin2022nernst,
  title={Nernst-Planck-Navier-Stokes systems near equilibrium},
  author={Constantin, Peter and Ignatova, Mihaela and Lee, Fizay-Noah},
  journal={Pure Appl. Funct. Anal},
  volume={7},
  number={1},
  pages={175--196},
  year={2022}
}

@article{constantin2021nernst,
  title={Nernst--Planck--Navier--Stokes systems far from equilibrium},
  author={Constantin, Peter and Ignatova, Mihaela and Lee, Fizay-Noah},
  journal={Archive for Rational Mechanics and Analysis},
  volume={240},
  number={2},
  pages={1147--1168},
  year={2021},
  publisher={Springer}
}

@article{constantin2021interior,
  title={Interior Electroneutrality in Nernst--Planck--Navier--Stokes Systems},
  author={Constantin, Peter and Ignatova, Mihaela and Lee, Fizay-Noah},
  journal={Archive for Rational Mechanics and Analysis},
  volume={242},
  number={2},
  pages={1091--1118},
  year={2021},
  publisher={Springer}
}

@article{gajewski1986basic,
  title={On the basic equations for carrier transport in semiconductors},
  author={Gajewski, Herbert and Gr{\"o}ger, Konrad},
  journal={Journal of mathematical analysis and applications},
  volume={113},
  number={1},
  pages={12--35},
  year={1986},
  publisher={Elsevier}
}

@article{arnold2001convex,
  title={On convex Sobolev inequalities and the rate of convergence to equilibrium for Fokker-Planck type equations},
  author={Arnold, Anton and Markowich, Peter and Toscani, Giuseppe and Unterreiter, Andreas},
  journal={Communications in Partial Differential Equations },
  volume={26},
  year={2001},
  publisher={Taylor \& Francis}
}

\end{document}